\documentclass[12pt]{amsart}
\usepackage{latexsym}
\usepackage{amsthm}
\usepackage{amssymb}
\usepackage[utf8]{inputenc}
\usepackage[T1]{fontenc}
\usepackage[all]{xy}
\usepackage{amsfonts}
\usepackage{amsmath}
\usepackage{graphicx}
\usepackage{pstricks}
\usepackage{url}
\usepackage{hyperref}
\usepackage[a4paper]{geometry}
\author{Aboubacar Nibirantiza} 
\address[Aboubacar Nibirantiza]{University of Burundi, Institute of Applied Pedagogy, Department of mathematics, B.P 2523, Bujumbura-Burundi}
\email{aboubacar.nibirantiza[at].ub.edu.bi}
\date{\today} 
\title[On the matrix realization of the Lie superalgebra $\mathfrak{spo}(2l+2|n)$]{On the matrix realization of the Lie superalgebra of contact projective vector fields $\mathfrak{spo}(2l+2|n)$}

\newtheorem{theorem}{Theorem}[section]

\newtheorem{df}[theorem]{Definition}

\newtheorem{rmk}[theorem]{Remark}

\newtheorem{prop}[theorem]{Proposition}

\newcommand{\R}{\mathbb{R}}

\newcommand{\N}{\mathbb{N}}

\newcommand{\id}{\mathrm{id}}

\newcommand{\spo}{\mathfrak{spo}}
\newcommand{\pgl}{\mathfrak{pgl}}
\newcommand{\gl}{\mathfrak{gl}}
\newcommand{\cK}{\mathcal{K}}		

\begin{document}

\begin{abstract}
In this paper, we show that the Lie superalgebra $\mathfrak{spo}(2l+2|n)$ is into the intersection of Lie superalgebra of contact vector fields $\mathcal{K}(2l+1|n)$ and the Lie superalgebra of projective vector fields $\mathfrak{pgl}(2l+2|n)$. We use mainly the embedding used by P. Mathonet and F. Radoux in "\textit{ Projectively equivariant quantizations over superspace $\mathbb{R}^{p|q}$. Lett. Math. Phys, 98: 311-331, 2011}". Explicitly, we use  the embedding of a Lie superalgebra constituted of matrices belonging to $\gl(2l+2|n)$ into $\mathrm{Vect}(\R^{2l+1|n})$.  We generalize thus in superdimension $2l+1-n$, the matrix realization described in \cite{MelNibRad13} on $S^{1|2}$. We mention that the intersection $\spo(2l+2|n)=\pgl(2l+2|n)\cap\cK(2l+1|n)$ that we prove here, in super case, has been prooved on $\R^{2l+2}$ in even case in \cite{CoOv12}. 
\end{abstract}


\maketitle

\section{Introduction}
The present paper is based on the concepts of supergeometry. It begins with a brief introduction to the notions that we need in the all sections, i.e: superfunctions, vector fields, differential $1$-superforms, etc on the superspace $\R^{m|n}$, where $m$ and $n$ are integers. We describe the supergeometry of $\R^{m|n}$ by its supercommutative superalgebra of superfunctions $C^{\infty}(\R^{m|n})$.\\

Using the standard contact structure on $\R^{2l+1|n}$, where $l$ is also an integer, we compute the formula of the contact vector fields on $\R^{2l+1|n}$. This formula is a generalization of those formulas known in classical geometry, as in \cite{CoOv12} and in supergeometry in low dimensions, as in \cite{GarMelOvs07, Mel09,MelNibRad13}. We also compute, in the super case, the formula of the Lagrange bracket of the superfunctions $f$ and $g$.\\

 As in \cite{MelNibRad13}, we consider an superskewsymmetric form $\omega$ defined on the superspace $\R^{2l+2|n}$ and we realize thus a Lie superalgebra $\spo(2l+2|n)$ constituted by the matrices $A$ of $\gl(2l+2|n)$ which preserve the form $\omega$. We use the method used by P. Mathonet and F. Radoux in \cite{MatRad11}. This construction allows us to embed the Lie superalgebra $\spo(2l+2|n)\subset\pgl(2l+2|n)$ into the Lie superalgebra $\mathrm{Vect}(\R^{2l+1|n})$ of vector fields on $\R^{2l+1|n}$. \\
 
 Thanks to the formula of contact vector fields $X_f$ obtained, for a certain superfuncfion given $f\in C^\infty(\R^{2l+1|n})$ of degree to most equal two in $z,x_i,y_i$ and $\theta_i$ variables, and to the formulas of projective vector fields of $\spo(2l+2|n)\subset\pgl(2l+2|n)$ obtained, we realize that the Lie superalgebra $\spo(2l+2|n)$ is constituted by the contact projective vector fields, i.e: the Lie superalgebra $\spo(2l+2|n)$ is into the intersection of the Lie superalgebra of projective vector fields $\pgl(2l+2|n)$ and the Lie superalgebra of contact vector fields $\cK(2l+1|n)$. To justify the terminology of contact projective vector fields for the elements of $\spo(2l+2|n)$, we refer to \cite{CoOv12}.

\section{Superfunctions on $\R^{2l+1|n}$}
We define the geometry of the superspace $\R^{2l+1|n}$, where $l\in\N,n\in\N^*$, by describing its associative supercommutative superalgebra of superfunctions on $\R^{2l+1|n}$ which we denote by $$C^\infty(\R^{2l+1|n}):=C^\infty(\R^{2l+1})\otimes \Lambda\R^n$$ and which is constituted by the elements 
 \begin{align*}
f(x,\theta)&=\sum_{0\leqslant|I|\leqslant n}{f_I(x)\theta_I}\\
&=f_0(x)+f_1(x)\theta_1+...+f_n(x)\theta_n+f_{12}(x)\theta_1\theta_2+...
+f_{1...n}(x)\theta_1...\theta_n
\end{align*}
where $|I|$ is the length of $I$, $x=(x_i),\quad i=1,\cdots, 2l+1$ is a coordinates system on $\R^{2l+1}$ and where $\theta=(\theta_i),\quad i=1,\cdots, n$ is odd Grassmann coordinates on $\Lambda\R^n$, i.e. $\theta_i^2=0,\quad \theta_i\theta_j=-\theta_j\theta_i$. We define the parity function $\tilde{.}$ by setting $\tilde{x}=0$ and $\tilde{\theta}=1$.

\section{Vector fields on $\R^{2l+1|n}$}
A vector fields on $\R^{2l+1|n}$ is a superderivation of the associative supercommutative superalgebra $C^\infty(\R^{2l+1|n})$. In coordinates, it can be expressed as 
\[
X=\sum_{i=1}^{2l+1}X^i\partial_{x_i}+\sum_{j=1}^nY^j\partial_{\theta_j},
\]
where $X^i$ and $Y^j$ are the elements of  $C^\infty(\R^{2l+1|n})$, $\partial_{x_i}=\frac{\partial}{\partial x_i}$ and $\partial_{\theta_j}=\frac{\partial}{\partial \theta_j}$ for all $i=1,2,\cdots, 2l+1$ and $j=1,2,\cdots,n$. \\
It can also be expressed as 
\[
X=\sum_{i=1}^{p+q}X^i\partial_{z_i},\] where $z_i=x_i$ for all $i\in\{1,\ldots,2l+1\}$ and $z_i=\theta_{i-(2l+1)}$ for all $i\in\{2l+2,\ldots, 2l+1+n\}$.
The parity function $\tilde{.}$ on vector field $X$ is defined as 
\[
\widetilde{\partial_{x_i}}=0\quad \mbox{and}\quad \widetilde{\partial_{\theta_i}}=1.
\]
The superspace of all vector fields on $\R^{2l+1|n}$ is a Lie superalgebra, which we shall denote by $\mathrm{Vect}(\R^{2l+1|n})$, by defining the following Lie bracket
\[
[X,Y]=XY-(-1)^{\tilde{X}\tilde{Y}}YX,
\] for all vector fields $X,Y$.
\section{Differential $1$-superforms on $\R^{2l+1|n}$}
We define the superspace $\Omega^1(\R^{2l+1|n})$ of differential $1$-superforms on $\R^{2l+1|n}$ as a superspace which is constituted by the elements 
\[ \alpha=\sum_{i=1}^{2l+1}{f_i(x_i,\theta_i)dx^i}+\sum_{i=1}^n{g_i(x_i,\theta_i)d\theta^i},\]
where $f_i$ and $g_i$ are elements of $C^\infty(\R^{2l+1|n})$ and $\widetilde{dx^i}=0,\,\widetilde{d\theta^i}=1$ and
where we set $\mathcal{B}'=(dx^i,d\theta^i)$ of $\Omega^1(\R^{2l+1|n})$ the dual basis of a basis $\mathcal{B}=(\partial_{x_i},\partial_{\theta_i})$ of $\mathrm{Vect}(\R^{2l+1|n})$ such that 
\[
\langle\partial_{x_j},dx^i\rangle=\delta^i_j,\quad \langle\partial_{x_j},d\theta^i\rangle=0 \quad\mbox{and}\quad \langle\partial_{\theta_j},d\theta_i\rangle=-\delta^i_j.
\]
\begin{rmk}
These elements $f_i$ and $g_i$ can also be declared at right and in this case we must use the even sign rule known in supergeometry.
\end{rmk}
When we consider a vector field $X$, we can also define the evaluation of differential $1$-superform on $X$, or the interior product of a differential $1$-superform $\alpha$ by $X$ as follow:
\[\alpha(X)=(-1)^{\tilde{X}\tilde{\alpha}}\langle X,\alpha\rangle,\quad\mbox{and}\quad i(X)\alpha=\langle X,\alpha\rangle.\]
Explicitly, if $X=\sum_{i=1}^{2l+1+n}X^i\partial_{z_i}$ and $\alpha=\sum_{j=1}^{2l+1+n}\alpha_jdz_j$, we have via the sign rule,
\begin{equation}\label{coucou1}
\langle X,\alpha\rangle=\langle\sum_{i=1}^{2l+n+1}X^i\partial_{z_i},\sum_{j=1}^{2l+n+1}\alpha_jdz_j\rangle=
\sum_{i,j=1}^{2l+n+1}X^i\alpha_j(-1)^{\tilde{i}\tilde{\alpha_j}}
\langle\partial_{z_i},dz_j\rangle=\sum_{i=1}^{2l+n+1}(-1)^{\tilde{i}(\tilde{\alpha_i}+\tilde{i})}X^i\alpha_i.
\end{equation}
We can generalize the definition of differential superforms and we have also a version of de de Rham differential which is adapted in the framework of supergeometry. Thus it allows us to define the Lie derivative of differential superforms. These operators have the analogue properties known in classical geometry.
\section{Standard contact structure on $\R^{2l+1|n}$}
We consider here the standard contact structure on $\R^{2l+1|n}$. We can find in \cite{Gr13} the notions of the contact structure on any supermanifold of dimension $m|n$.
\begin{df}
The standard contact structure on $\R^{2l+1|n}$ is defined by the kernel of the differential $1$-superforms $\alpha$ on $\R^{2l+1|n}$ which, in the system of Darboux coordinates $(z,x_i,y_i,\theta_j),\quad i=1,\cdots,l$ and $j=1,\cdots,n$ it can be written as 
\begin{equation}\label{STANDARDCONTACT}
\alpha=dz+\sum_{i=1}^l{(x_idy_i-y_idx_i)}+\sum_{i=1}^n{\theta_id\theta_i}.
\end{equation}
This differential $1$-superform $\alpha$ is called contact form on $\R^{2l+1|n}$ and we denote by $\mathrm{Tan}(\R^{2l+1|n})$ the space constituted of the elements of the kernel of $\alpha$. 
\end{df}
If we denote $q^A=(z,q^r)$ the generalized coordinate where 
\begin{equation}\label{gencoordin}
q^A=\left\lbrace 
\begin{array}{lcl}
 z\quad\mbox{if}\quad A=0\\
 x_A \quad\mbox{if}\quad 1\leqslant A\leqslant l,\\ 
 y_{A-l} \quad\mbox{if} \quad l+1\leqslant A\leqslant 2l\\ 
 \theta_{A-2l} \quad\mbox{if}\quad 2l+1\leqslant A\leqslant 2l+n 
\end{array}\right.
\end{equation}
we can write $\alpha$ in the following way
\[
\alpha=dz+\omega_{rs}q^rdq^s, \quad (\omega_{rs})=\left(
\begin{array}{cc|c}
0& \id_l&0\\
-\id_l&0&0\\
\hline
0&0&\id_n
\end{array}
\right).
\]
\begin{rmk}
We denote by $\omega^{sk}$ the matrix so that $(\omega_{rs})(\omega^{sk})=(\delta^k_r)$. We have thus
\[
(\omega^{rs})=\left(
\begin{array}{cc|c}
0&- \id_l&0\\
\id_l&0&0\\
\hline
0&0&\id_n
\end{array}
\right).
\] and $(\omega^{rs})=-(-1)^{\tilde{r}\tilde{s}}(\omega^{sr}).$
\end{rmk}
\begin{df}
We call the field of Reeb  on $\R^{2l+1|n}$, the vector field $T_0\in\mathrm{Vect}(\R^{2l+1|n})$ which, in the system of Darboux coordinates, one write 
$T_0=\partial_z$.
\end{df}
We can show that the field of Reeb is the unique vector field on $\R^{2l+1|n}$ so that $i(T_0)\alpha=1$ and $i(T_0)d\alpha=0$.\\
\begin{prop}
In the system of Darboux coordinates, the elements $T_r$ of $\mathrm{Tan}(\R^{2l+1|n})$ can be written as follow
\begin{equation}\label{tangentdistr}
T_r=\left\lbrace 
\begin{array}{lcl} 
 A_r &:=& \partial_{x_r}+y_r\partial_z\quad\mbox{if}\quad 1\leqslant r\leqslant l\\ 
 -B_{r-l} &:=& \partial_{y_{r-l}}-x_{r-l}\partial_z\quad\mbox{if}\quad l+1\leqslant r\leqslant 2l\\ 
 \overline{D}_{r-2l}&:=& \partial_{\theta_{r-2l}}-\theta_{r-2l}\partial_z \quad\mbox{if}\quad 2l+1\leqslant r\leqslant 2l+n
\end{array}\right.
\end{equation}
\end{prop}
\begin{proof}
If we denote by $T_r$ the vector field $T_r=\partial_{q^r}-\langle\partial_{q^r},\alpha\rangle\partial_z,$ and since $\widetilde{\alpha}=0$, we have
\[\alpha(T_r)=\langle T_r,\alpha\rangle=
\langle \partial_{q^r},\alpha\rangle-
\langle\langle\partial_{q^r},\alpha\rangle\partial_z,\alpha\rangle=\langle \partial_{q^r},\alpha\rangle-\langle \partial_{q^r},\alpha\rangle=0.\]
We can also show that any vector field $X$ of $\mathrm{Tan}(\R^{2l+1|n})$ can be written as a linear combination of the vector fields $T_r$. It is useful to compute the vector fields $T_r$ according to the matrix $\omega$. One has, via \ref{coucou1}, \[
T_r=\partial_{q^r}-\alpha_r\partial_z=\partial_{q^r}-\omega_{kr}q^k\partial_z.
\]
It is sufficient to vary $r$ in the interval $[1, 2l+n]$ to conclude.
\end{proof}
The following formulas are immediate.
\begin{equation}\label{TANGDIST}
T_r(q^k)=\delta_r^k,\quad T_r(z)=-\omega_{kr}q^k,\quad [T_r,T_j]=-2\omega_{rj}\partial_z,\quad T_r(z^2)=-2z\omega_{kr}q^k.
\end{equation}

\section{Contact vector fields on $\R^{2l+1|n}$}\label{coucou2}
\begin{df}
We call a contact vector field on $\R^{2l+1|n}$ a vector field $X$  that preserves the contact structure, i.e. a vector field $X$ verifying the following condition: $[X,T]\in \mathrm{Tan}(\R^{2l+1|n}) $ for all $T\in \mathrm{Tan}(\R^{2l+1|n})$.
\end{df}
The following proposition is known in the classical geometry \cite{CoOv12} and in supergeometry in small dimensions, i.e: $(1|1)$ and $(1|2)$ in \cite{MelNibRad13,Mel09, GarMelOvs07}. We give here its analogue in supergeometry and generalize it in dimension $(m|n)$. It is our main first result.
\begin{prop}
A vector field $X$ on $\R^{2l+1|n}$ is called contact vector field if and only if it exists a  superfunction $f$ that $X=X_f$ where $X_f$ is given by the following formula
\begin{equation}\label{CONTACTFORME}
X_f=f\partial_z-\frac{1}{2}(-1)^{\tilde{f}\tilde{T_r}}\omega^{rs}T_r(f)T_s,
\end{equation} 
We denote by $\cK(2l+1|n)$ the space of the all contact vector fields on $\R^{2l+1|n}$.
\end{prop}
\begin{proof}
Seen the definition of $T_r$, we can say that any vector field on $\R^{2l+1|n}$ can be written as $X=f\partial_z+\sum_{i=1}^{2l+n}g_iT_i$. The vector field $X$ is thus called contact vector field if and only if 
\[
[f\partial_z+\sum_{i=1}^{2l+n}{g_iT_i},T_j]\in\,<T_1,\cdots,T_{2l+n}>,\quad\forall j\in
\{1,\cdots,2l+n\}.
\]
This formula can be also written as
\begin{eqnarray*}
[f\partial_z+\sum_{i=1}^{2l+n}{g_iT_i},T_j]&=&-(-1)^{\tilde{f}\tilde{T}_j}[T_j,f\partial_z]-(-1)^{(\tilde{g}_i+\tilde{T}_i)\tilde{T}_j}\sum_{i=1}^{2l+n}{[T_j,g_iT_i]}\\
&=&-(-1)^{\tilde{f}\tilde{T}_j}T_j(f)\partial_z+(-1)^{\tilde{T}_i\tilde{T}_j}2\sum_{i=1}^{2l+n}{g_i\omega_{ji}\partial_z}-(-1)^{(\tilde{g}_i+\tilde{T}_i)\tilde{T}_j}\sum_{i=1}^{2l+n}{T_j(g_i)T_i}.
\end{eqnarray*}
This vector field $X$ is in the kernel of $\alpha$ if and only if
\[
-(-1)^{\tilde{f}\tilde{T}_j}T_j(f)-2\sum_{i=1}^{2l+n}g_i\omega_{ij}=0,
\]
for all $j\in \{1,\cdots, 2l+n\}$. This equation shows that all proposed vector fields $X_f$ are contact vector fields. In the other hand, this equation implies also that
\[
-(-1)^{\tilde{f}\tilde{T}_j}T_j(f)\omega^{jk}-2\sum_{i=1}^{2l+n}{g_i\omega_{ij}\omega^{jk}}=0, \quad\forall j\in\{1,\cdots,2l+n\},
\]
or, when we sum on j, we have
\[
-(-1)^{\tilde{f}\tilde{T}_j}\omega^{jk}T_j(f)=2\sum_{i=1}^{2l+n}{g_i\omega_{ij}\omega^{jk}}=2\sum_{i=1}^{2l+n}{g_i\delta_i^k}.
\]
We obtain directly that
\[
g_k=-\frac{1}{2}(-1)^{\tilde{f}\tilde{T}_j}\omega^{jk}T_j(f)
\]
and this allows us to conclude.

\end{proof}
The following proposition gives, in the super case, the formula of the Lagrange bracket of the superfunctions $f$ and $g$. It is the generalization of the formula given in \cite{Mel09, MelNibRad13, GarMelOvs07}.
\begin{prop}
The set $\cK(2l+1|n)$ is a Lie sub superalgebra of $\mathrm{Vect}(\R^{2l+1|n})$.	More explicitly, if $X_f$ and $X_g$ are the elements of $\cK(2l+1|n)$, one writes
\begin{equation}\label{LAGRANGE}
[X_f,X_g]=X_{\{f,g\}}
\end{equation} where the superfunction $\{f,g\}$ is given by
\begin{equation}\label{LagrBracket}
{\{f,g\}}:=fg'-f'g-\frac{1}{2}(-1)^{\tilde{T}_r\tilde{f}}\omega^{rs}T_r(f)T_s(g),
\end{equation} and where $h'=\partial_z(h)$.
\end{prop}
\begin{proof}
The Lie bracket $[X_f,X_g]$ of the two contact vector fields $X_f$ and $X_g$ is also a contact vector field. Indeed, the Lie bracket 
\[
[X_f,X_g]=[f\partial_z-\frac{1}{2}(-1)^{\tilde{T}_r\tilde{f}}\omega^{rs}T_r(f)T_s,g\partial_z-\frac{1}{2}(-1)^{\tilde{T}_k\tilde{g}}\omega^{kl}T_k(g)T_l] 
 \] is written as
  \begin{multline*}
[f\partial_z,g\partial_z]-\frac{1}{2}(-1)^{\tilde{T}_k\tilde{g}}\omega^{kl}[f\partial_z,T_k(g)T_l]
-\frac{1}{2}(-1)^{\tilde{T}_r\tilde{f}}\omega^{rs}[T_r(g)T_s,g\partial_z]\\
+\frac{1}{4}\omega^{rs}\omega^{kl}[T_r(f)T_s,T_k(g)T_l].
 \end{multline*}
 The sum of the first three Lie brackets equals to
  \begin{multline*}
 (fg'-f'g)\partial_z+\frac{1}{2}(-1)^{\tilde{g}(\tilde{T}_k+\tilde{f})}\omega^{kl}T_k(g)T_l(f)\partial_z-\frac{1}{2}(-1)^{\tilde{f}\tilde{T}_r}\omega^{rs}T_r(f)T_s(g)\partial_z\\
 -\frac{1}{2}(-1)^{\tilde{g}\tilde{T}_k}\omega^{kl}fT_k(g')T_s+\frac{1}{2}(-1)^{\tilde{f}\tilde{T}_r+\tilde{f}\tilde{g}}\omega^{rs}gT_r(f')T_s
 \end{multline*}
 and the fourth Lie bracket equals to 
 \[
\frac{1}{4}(-1)^{\tilde{f}\tilde{T}_r+\tilde{g}\tilde{T}_k}\omega^{rs}\omega^{kl}\left( 
T_r(f)T_sT_k(g)T_l-(-1)^{\tilde{f}\tilde{g}}T_k(g)T_lT_r(f)T_s 
\right) -\frac{1}{2}(-1)^{(\tilde{T}_r+\tilde{f})\tilde{g}}\omega^{kr}T_k(g)T_r(f)\partial_z.
 \]

Since the Lie bracket of two contact vector fields is also a contact vector field and since $X_{\{f,g\}}$ is written, via the formula \eqref{CONTACTFORME}, by 
\[
(\{f,g\})\partial_z-\frac{1}{2}(-1)^{(\tilde{f}+\tilde{g})\tilde{T}_r}\omega^{rs}T_r(\{f,g\})T_s,
 \]
 then we can see that the sum of the coefficients of $\partial_z$ gives the formula of Lagrange bracket. One has thus
 \[
\{f,g\}= fg'-f'g-(-1)^{\tilde{f}\tilde{T}_r}\frac{1}{2}\omega^{rs}T_r(f)T_s(g).
 \] 
\end{proof}
Via the Lagrange formula \eqref{LAGRANGE}, the Lie bracket of contact vector fields, which defines a Lie superalgebra structure on $\cK(2l+1|n)$, induces a Lie  superalgebra structure on the superspace $C^\infty(\R^{2l+1|n})$ by the bilinear law given by \eqref{LagrBracket}.\\
The following remark is very important:
\begin{rmk}
The Lagrange bracket of superfunctions $f$ and $g$ \, of degree to most equal two
 is always a superfunction of degree to most equal two.
\end{rmk}

This remark allows us to define the Lie superalgebra constituted by the contact vector fields $X_f$ which the associated superfunctions $f$ are of degrees to most equal two in $z,x_i,y_i$ and $\theta_i$ variables. We denote this Lie superalgebra temporarily by $\mathfrak{g}\subset\cK(2l+1|n)$.

\section{Matrix realization of $\spo(2l+2|n)$}
In this section, we embed a Lie sub-superalgebra $\spo(2l+2|n)$ of $\gl(2l+2|n)$ in the Lie superalgebra $\mathrm{Vect}(\R^{2l+1|n})$. We use the method used in \cite{MatRad11} and we show that the Lie  superalgebra obtained is exactly isomorphic to $\mathfrak{g}$.\\

We consider a matrix $G$ defined by $G=\begin{pmatrix}
J&0\\0&\id_n
\end{pmatrix} $ such that $J=\begin{pmatrix}
0&-\id_{l+1}\\\id_{l+1}&0
\end{pmatrix}$.  We define on $\R^{2l+2|n}$ the following superskewsymmetric form $\omega$ associated to the matrix $G$ as

\begin{equation}
\omega:\R^{2l+2|n}\times\R^{2l+2|n}\to\R:(U,V)\to V^tGU,
\end{equation} where $A^t$ is the usual transpose of the matrix $A$.

\begin{df}
We define a Lie superalgebra $\spo(2l+2|n)$ constituted by the matrices $A$ of $\gl(2l+2|n)$ which preserve the form $\omega$, i.e such that
\begin{equation}\label{formeOmega}
\omega(AU,V)+(-1)^{\tilde{A}\tilde{U}}\omega(U,AV)=0,\quad \forall U,V\in\R^{2l+2|n}.
\end{equation}
\end{df}
Our second main result is the following:
\begin{theorem}
The Lie superalegbra $\spo(2l+2|n)$ is the space of the matrices $A=\begin{pmatrix}
A_1&A_2\\A_3&A_4
\end{pmatrix}$ that the blocks $A_1,A_2,A_3$  and $A_4$ satisfy the following conditions
\begin{enumerate}
\item $A_1^tJ+JA_1=0,i.e: A_1\in \mathfrak{sp}(2l+2)$\\
\item $A_4^t+A_4=0, i.e: A_4\in \mathfrak{o}(n)$\\
\item $A_3=-A_2^tJ$.
\end{enumerate}
\end{theorem}
\begin{proof}
We consider the following matrices
\[A=\begin{pmatrix} A_1& A_2\\A_3&A_4\end{pmatrix}\in\mathfrak{gl}(2l+2|n);\quad\mbox{where}\quad A_1\in\mathfrak{gl}(2l+2), A_2\in\R^n_{2l+2},  A_3\in\R^{2l+2}_n,A_4\in\R^n_n.\]
For all vector fields $U=\begin{pmatrix}U_1\\U_2\end{pmatrix},V=\begin{pmatrix}V_1\\V_2
\end{pmatrix}$
of $\R^{2l+2|n}$, we compute the matrices $A$ of $\gl(2l+2|n)$ which satisfy \eqref{formeOmega}. First, we can see that the first term $\omega(AU,V)$ of \eqref{formeOmega} equals to \[V_1^tJA_1U_1+V_2^tA_3U_1+V_1^tJA_2U_2+V_2^tA_4U_2\] and the second term $(-1)^{\tilde{A}\tilde{U}}\omega(U,AV)$ equals to \[V_1^tA_1^tJU_1+V_2^tA_2^tJU_1-V_1^tA_3^tU_2+V_2^tA_4^tU_2.\]
It is also easy to see that the formula \eqref{formeOmega} equals to
\begin{multline}\label{formeOmega2}
V_1^tJA_1U_1+V_2^tA_3U_1+V_1^tJA_2U_2+V_2^tA_4U_2+V_1^tA_1^tJU_1\\
+V_2^tA_2^tJU_1-V_1^tA_3^tU_2+V_2^tA_4^tU_2=0, \quad\forall U,V\in\R^{2l+2|n}.
\end{multline}
In particular, if $U_2=0$, $V_2=0$, then the equation \eqref{formeOmega2} equals to
\[V_1^t(JA_1+A_1^tJ)U_1=0, i.e:JA_1+A_1^tJ=0.\] This last condition means that the blocks $A_1$ are symplectic matrices.\\
If we set $U_1=0$ and $V_1=0$, then the equation \eqref{formeOmega2} becomes
 \[V_2^t(A_4^t+A_4)U_2=0,i.e: A_4^t+A_4=0.\] This condition means that the block $A_4$ is an orthogonal matrix.\\
 Finally, if $U_2=0$ and $V_1=0$, then the equation \eqref{formeOmega2} equals to
 \[V_2^t(A_3+A_2^tJ)U_1=0,i.e: A_3+A_2^tJ=0.\]
\end{proof}

In the following, we describe a basis of $\spo(2l+2|n)$. If we denote by $a_{i,j}$ the number $a\in\R$ situated on the $i^{th}$ line and on the $j^{th}$ column, we can see that this basis is constituted by the three following types of matrices:\\
The first type of matrices of the basis of $\spo(2l+2|n)$ is associated to the symplectic algebra $\mathfrak{sp}(2l+2)$ and is given by following family of matrices:

\begin{multline}\label{BaseMatrix}
 \left(\begin{array}{cc|ccc|c}
 {}&{}&{}&{}&{}&{}\\
 1_{i,j}&{}&{}&0&{}&0\\
{}&{}&{}&{}&{}&{}\\
 \hline
 {}&{}&{}&{}&0&{}\\
 {}&{}&{}&{}&{}&{}\\
 0&{}&{}&-1_{(l+1+j),(l+1+i)}&{}&0\\
{}&{}&{}&{}&{}&{}\\
{}&{}&0&{}&{}&{}\\ 
\hline
0&{}&{}&0&{}& 0_{n,n}\\
\end{array}\right);
 \left(\begin{array}{cc|cc|c}
 {}&{}&0&1_{i,(l+1+j)}&{}\\
 0&{}&{}&{}&0\\
{}&{}&1_{j,(l+1+i)}&{}&{}\\
{}&{}&{}&0&{}\\
 \hline
 {}&{}&{}&{}&{}\\
 0&{}&0&{}&0\\
{}&{}&{}&{}&{}\\ 
\hline
{}&{}&{}&{}&{}\\
0&{}&0&{}&0_{n,n}\\
{}&{}&{}&{}&{}\\
\end{array}\right), \\ \left(\begin{array}{cc|cc|c}
 {}&{}&{}&{}&{}\\
0&{}&{}&0&0\\
{}&{}&{}&{}&{}\\
 \hline
0&1_{(l+1+i),j}&{}&{}&{}\\
 {}&{}&{}&{}&{}\\
1_{(l+1+j),i}&{}&{}&0&0\\
{}&0&{}&{}&{}\\
\hline
{}&{}&{}&{}&{}\\
0&{}&0&{}&0_{n,n}\\
{}&{}&{}&{}&{}\\
\end{array}\right)
\quad  1\leq i,j\leq l.
\end{multline} 

The second type of matrices is given by the following family of matrices:
\begin{multline}\label{BaseMatrix1}
\left(\begin{array}{cc|cc}
{}&{}&{}&{}\\
0&{}&{}&1_{(i,(2l+2+j))}\\
{}&{}&{}&{}\\
\hline
{}&{}&{}&{}\\
-1_{((2l+2+j),i-(l+1))}&{}&{}&0\\
{}&{}&{}&{}\\
\end{array}\right)\quad \mbox{if}\quad l+1\leq i\leq 2l+2,\quad 1\leq j\leq n\\
\quad\mbox{and}\quad
\left(\begin{array}{cc|cc}
{}&{}&{}&{}\\
0&{}&{}&1_{(i,(2l+2+j))}\\
{}&{}&{}&{}\\
\hline
{}&{}&{}&{}\\
1_{((2l+2+j),(l+1+i))}&{}&{}&0\\
{}&{}&{}&{}\\
\end{array}\right)\quad\mbox{if}\quad 1\leq i\leq l+1,\quad 1\leq j\leq n.
\end{multline}
And the third type is associated to the orthogonal algebra $\mathfrak{o}(n)$ and is given by:
\begin{equation}\label{BaseMatrix2}
\left(\begin{array}{cc|ccc}
{}&{}&{}&{}&{}\\
0&{}&{}&0&{}\\
{}&{}&{}&{}&{}\\
\hline
{}&{}&0&{}&{}\\
{}&{}&{}&{}&1_{((2l+2+i),(2l+2+j))}\\
0&{}&{}&{}&{}\\
{}&{}&-1_{((2l+2+j),(2l+2+i))}&{}&{}\\
{}&{}&{}&{}&0
\end{array}\right).
\end{equation}
Our third main result is given by the following theorem:
\begin{theorem}
The Lie superalgebra $\mathfrak{g}$ made in evidence at the end of the section \ref{coucou2} and whose superfunctions $f$ are degrees to most equal two is isomorphic to the Lie superalgebra $\spo(2l+2|n)$.
\end{theorem}
\begin{proof}
Because of $Id\notin \spo(2l+2|n)$ we can define the injective homomorphism 
 \begin{equation}\label{iota3}
\iota:\mathfrak{spo}(2l+2|n)\to\mathfrak{pgl}(2l+2|n):A\mapsto [A].
\end{equation}
Now, the Lie superalgebra $\pgl(2l+2|n)$ can be embedded into the Lie superalgebra of vector fields on $\R^{2l+1|n}$ thanks to the projective embedding define in \cite{MatRad11} in the following way:
\begin{equation}\label{plongProj}
\left[\left(\begin{array}{ll}0&\xi\\v&B\end{array}\right)\right]\mapsto-\sum_{i=1}^{2l+1+n}v^i
\partial_{t^i}-\sum_{i,j=1}^{2l+1+n}(-1)^{\tilde{j}(\tilde{i}+\tilde{j})}B_j^i t^j\partial_{t^i}+\sum_{i,j=1}^{2l+1+n}(-1)^{\tilde{j}}\xi_j t^j t^i\partial_{t^i},
\end{equation}
where $v\in \R^{2l+1|n}$, $\xi\in (\R^{2l+1|n})^*$, $B\in\mathfrak{gl}(2l+1|n)$ and the coordinates $t^{1},t^{2},\cdots,t^{2l+1+n}$ corresponds respectively to $x_1,\cdots, x_l, z,y_1,\cdots,y_l,\theta_{1},\cdots,\theta_{n}.$\\

Composing $\iota$ with the projective embedding, we can embed $\spo(2l+2|n)$ into $\mathrm{Vect}(\R^{2l+1|n})$. If we compute this embedding on the generators of $\spo(2l+2|n)$ written above, we obtain via \eqref{CONTACTFORME}, the contact projective vector fields $X_f$ for a certain given $f\in C^\infty(\R^{2l+1|n})$.\\
In the following, we study explicitly the three types of matrices described above.\\
For the first matrix of \eqref{BaseMatrix}, i.e: when $i,j\in[1,l]$, we obtain
 the following  contact projective vector fields.
\begin{enumerate}
\item If $i=j=1$, we have
\[x_i\partial_{x_i}+y_i\partial_{y_i}+\theta_i\partial_{\theta_i}+2z\partial_z\quad, i.e: \quad 2X_z \]\\ 
\item If $i=1$ and $j\neq 1$, then we have
\[ x_{j-1}(x_i\partial_{x_i}+y_i\partial_{y_i}+z\partial_z+\theta_i\partial_{\theta_i})+
 z\partial_{y_{j-1}},\quad i.e: \quad 2X_{x_{j-1}z}\]\\
\item If $i\neq 1$ and $j=1$, then we obtain
\[
-\partial_{x_{i-1}}+y_{i-1}\partial_z,\quad i.e: \quad
2X_{y_{i-1}}
\]  
\item If  $i\neq 1$ and $j\neq 1$, then \[
y_{i-1}\partial_{y_{j-1}}-x_{j-1}\partial_{x_{i-1}}, \quad i.e:\quad
2X_{x_{i-1}y_{i-1}}.
\]
\end{enumerate}
For the second matrix of \eqref{BaseMatrix}, we obtain the following contact projective vector fields
\begin{enumerate}
\item If $i=j=1$, one has 
\[
z(z\partial_z+x_i\partial_{x_i}+y_i\partial_{y_i}+\theta_j\partial_{\theta_j}), 
\quad i.e: \quad X_{z^2}
\] 
\item If $i=j$ and $j\neq 1$, we have 
\[ -y_{i-1}\partial_{x_{i-1}}, \quad i.e: \quad
X_{y^2_{i-1}},
\]
\item If $i\neq j$ and $i=1$, one has 
\[
 y_{j-1}(x_i\partial_{x_i}+y_i\partial_{y_i}+z\partial_z+\theta_j\partial_{\theta_j})-z\partial_{x_{j-1}}, \quad i.e:\quad
2X_{y_{j-1}z}
\]
\item If $i\neq j$ and $i\neq 1$,  one has
\[
 -(y_{j-1}\partial_{x_{i-1}}+y_{i-1}\partial_{x_{j-1}}), \quad i.e:\quad
2X_{y_{i-1}y_{j-1}}.
\]
\end{enumerate}
For the third matrix of \eqref{BaseMatrix}, we obtain the following contact projective vector fields
\begin{enumerate}
\item If $i=j=1$, we obtain

\[ -\partial_z,   \quad i.e: \quad-X_1 \] 
\item If $i=j$ and $j\neq 1$, we have
\[  -x_{i-1}\partial_{y_{i-1}},  \quad i.e: \quad -X_{x^2_{i-1}} 
\]
\item If $i\neq j$ and $i=1$,  we write
\[ -(\partial_{y_{j-1}}+x_{j-1}\partial_z),  \quad i.e: \quad -2X_{x_{j-1}} 
\]
\item If $i\neq j$ and $i\neq 1$, then we have
\[-(x_{j-1}\partial_{y_{i-1}}+x_{i-1}\partial_{y_{j-1}}),  \quad i.e:\quad -2X_{x_{j-1}x_{i-1}}.
\]
\end{enumerate}
Now, we study explicitly any matrix of \eqref{BaseMatrix1}, i.e: when $l+1\leqslant i\leqslant 2l+2$ and $1\leqslant j \leqslant n$.\\
For the first matrix of \eqref{BaseMatrix1}, we obtain the following contact projective vector fields 
\begin{enumerate}
\item If $i=l+2$, then we have
\[
 \theta_j\partial_z+\partial_{\theta_j},  \quad i.e:\quad 2X_{\theta_j}
\]
\item and if $i\neq l+2$, we obtain
\[
\theta_j\partial_{y_{i-l-2}}+x_{i-l-2}\partial_{\theta_j}, \quad i.e:\quad 2X_{x_{i-l-2}\theta_j}.
\] 
\end{enumerate}
For the second matrix of \eqref{BaseMatrix1}, we have the following contact projective vector fields
\begin{enumerate}
\item If $i=1$, one has
\[
-\theta_j(x_i\partial_{x_i}+y_i\partial_{y_i}+z\partial_z+
\theta_j\partial_{\theta_j})-z\partial_{\theta_j}, \quad i.e:\quad -2X_{z\theta_j}
\]
\item and if $i\neq 1$, then we obtain
\[
 \theta_j\partial_{x_{i-1}}-y_{i-1}\partial_{\theta_j}, \quad i.e:\quad -2X_{y_{i-1}\theta_j}.
\]
\end{enumerate}
Finally, for the matrix \eqref{BaseMatrix2}, i.e: when $2l+2\leqslant i,j\leqslant 2l+2+n$, we obtain the contact projective vector field
\[
 \theta_{i-1}\partial_{\theta_{j-1}}-\theta_{j-1}\partial_{\theta_{i-1}}, i<j ,  \quad i.e:\quad 2X_{\theta_{i-1}\theta_{j-1}}.
\]
The Lie superalgebra $\mathfrak{g}$ is isomorphic to $\spo(2l+2|n)$ thanks to the identification of the generators of these two Lie superalgebras.
\end{proof}
\begin{rmk}
The Lie superalgebra $\spo(2l+2|n)$ is into the intersection of the Lie superalgebra of contact vector fields $\cK(2l+1|n)$ and the Lie superalgebra of projective vector fields $\pgl(2l+2|n)$. Thus, as in \cite{CoOv12}, the elements $X_f$ of $\spo(2l+2|n)$ are called contact projective vector fields.
\end{rmk}


\begin{thebibliography}{10}
\bibitem{Lei80}
D.~A.~Leites.
\newblock { Introduction to the theory of Supermanifolds}. 
\newblock Russian Math.Surveys 35:1 (1980),1-64.

\bibitem{Ber87}
F.~A.~Berezin.
\newblock { Introduction to superanalysis}, volume~9 of {\em Mathematical
  Physics and Applied Mathematics}.
\newblock D. Reidel Publishing Co., Dordrecht, 1987.
\newblock Edited and with a foreword by A. A. Kirillov, With an appendix by V.
  I. Ogievetsky, Translated from the Russian by J. Niederle and R. Koteck{\'y},
  Translation edited by D.~Leites.
  
 \bibitem{KacAdvances}
V.~G. Kac.
\newblock Lie superalgebras.
\newblock {\em Advances in Math.}, 26(1):8--96, 1977.

\bibitem{CoOv12}
C.~Conley and V.~Ovsienko.
\newblock Linear Differential Operators on Contact manifolds.
\newblock {\em arxiv:math-Ph/1205.6562v1},24p, 2012.

\bibitem{LPS}
D.~Leites, E.~Poletaeva and V.~Serganova.
\newblock On Einstein Equations on Manifolds. and Supermanifolds.
\newblock {\em J. of Nonlinear Math. Phys.}, 9(4):394-425, 2002, math.DG/0306209

\bibitem{MatRad11}
P.~Mathonet and F.~Radoux.
\newblock Projectively equivariant quantizations over the Superspace $\R^{p|q}$.
\newblock {\em Lett. Math. Phys.}, 98:311--331, 2011.

  \bibitem{MelNibRad13}
    N.~Mellouli, and A. Nibirantiza,  and F. Radoux.
    \newblock{$\spo(2|2)$-Equivariant Quantizations on the Supercircle $S^{1|2}$}.   \newblock{SIGMA 9 (2013), 055, 17 pages}
     \newblock{ http://dx.doi.org/10.3842/SIGMA.2013.055 }, arxiv:math.DG/1302.3727v2
\bibitem{Nib14}
A.~Nibirantiza.
\newblock {\em Sur les quantifications équivariantes en supergéométrie de contact}.
\newblock {Th\`ese de Doctorat, Universit\'e de Li\`ege.}
 http://hdl.handle.net/2268/169770, 2014.
 
\bibitem{GarMelOvs07}
H.~Gargoubi, N.~Mellouli, and V.~Ovsienko.
\newblock {Differential operators on supercircle: conformally equivariant
  quantization and symbol calculus.}
\newblock {\em Lett. Math. Phys.}, 79(1):51--65, 2007.

\bibitem{Gr13}
    ~J.Grabowski.
     \newblock{Graded contact manifolds and contact courant algebroids.}
   \newblock{Journal of Geometry and Physics},68(2013),27-58.
\bibitem{Mel09}
N.~Mellouli.
\newblock Second-order conformally equivariant quantization in dimension $1|2$.
\newblock {\em SIGMA}, 5(111), 2009.

\end{thebibliography}
\end{document}